\documentclass[12pt,reqno]{amsart}
\usepackage{amscd,amssymb,amsmath,amsthm}
\usepackage[table]{xcolor}
\usepackage{multirow}
\usepackage{caption} 
\usepackage{array, tabularx, caption, boldline}

\usepackage{booktabs}
\usepackage{cellspace}
\usepackage{cite}
\usepackage{tikz}
\usetikzlibrary {positioning}
\usetikzlibrary{arrows}
\usepackage{multirow}
\newcommand{\bydef}{:=}
\usepackage{amssymb,upref}
\usepackage[mathcal]{euscript}
\usepackage[cmtip,all]{xy}
\usepackage{mathtools}
\newtheorem{con}[subsection]{Conjecture}
\newtheorem{thm}[subsection]{Theorem}
\newtheorem{lemma}[subsection]{Lemma}
\newtheorem{pro}[subsection]{Proposition}
\newtheorem{cor}[subsection]{Corollary}
\newenvironment{romanenumerate}
{\begin{enumerate}
 
 }{\end{enumerate}}

\usepackage{tikz}
\usepackage{multirow}
\usetikzlibrary{arrows}
\newtheorem{rk}[subsection]{Remark}
\newtheorem{defn}[subsection]{Definition}
\newtheorem{ex}{Example}

\numberwithin{equation}{section} 

\newcommand{\cE}{{\mathcal E}}

\DeclareMathOperator{\ann}{ann}

\newcommand{\NN}{{\mathbb N}}
\newcommand{\FF}{\mathbb{F}}

\newcommand{\w}{{\bf w}}

\newcommand{\bea}{\begin{eqnarray}}
\newcommand{\eea}{\end{eqnarray}}

\def\ann{\operatorname{ann}}

\def\w{\omega}

\begin{document}

\title[Volterra Evolution Algebras and Their Graphs]
{Volterra Evolution Algebras and Their Graphs}
\author{Izzat Qaralleh}
\address{Izzat Qaralleh\\
Department of Mathematics\\
Faculty of Science, Tafila Technical
University\\
Tafila, Jordan}
\email{{\tt izzat\_math@yahoo.com}}

\author{Farrukh Mukhamedov}
\address{Farrukh Mukhamedov\\
 Department of Mathematical Sciences\\
College of Science, The United Arab Emirates University\\
P.O. Box, 15551, Al Ain\\
Abu Dhabi, UAE} \email{{\tt far75m@gmail.com} {\tt
farrukh.m@uaeu.ac.ae}}

\date{Received: xxxxxx; Revised: yyyyyy; Accepted: zzzzzz.
\newline \indent $^{*}$ Corresponding author}

\begin{abstract}
In this paper, we introduce Volterra evolution algebras which are evolution algebras whose structural matrices  are described by skew symmetric matrices.  A main result of the present paper gives a connection between such kind of
 algebras with ergodicities of Volterra  quadratic stochastic operators. Furthermore, some of properties of the considered algebras such as nilpotency, derivations have been studied as well.      
\end{abstract}

\maketitle

\section{Introduction}

There exist several classes of non-associative algebras (baric, evolution, Bernstein, train, stochastic,
etc.), whose investigations have provided a number of significant contributions to theoretical population
genetics \cite{Reed, WB}. These classes have been defined in different times by several authors, and all algebras
belonging to such classes are generally called \textit{genetic} (see \cite{ZSSS,WB}). In \cite{E}  it was introduced the formal language of
abstract algebra to the study of the genetics. On the other hand, problems of population genetics can be traced
back to Bernstein's work \cite{B0} where evolution operators were studied which naturally define genetic algebras (see \cite{CM,lu}).  Notice that such kind of evolution operators are described by quadratic stochastic operators (QSO) \cite{GMR,lu}. 
Dynamics of QSO is closely related to the 
investigation of certain algebraic properties of the evolution algebras (see \cite{R8}).    

In\cite{TV}  a new type of evolution algebra has been introduced. 
Thereafter, in \cite{T} the foundations of these algebras have been established. These types of algebras lie between algebras and dynamical systems. Although, evolution algebras do not form a variety (they are not defined by identities),
algebraically, their structure has table of multiplication, which satisfies the conditions of
commutative algebra. Dynamically, they represent discrete dynamical systems. In this context, an evolution algebra is nothing
but a finite-dimensional algebra $\cE$ provided with a basis $B = \{e_i : \  i\in \Lambda \}$ , such that $e_ie_j=0$ whenever $i\neq j$   
(such a basis is said to be \textit{natural}), and $e_i^2=\sum_{k=1}^{n}a_{ik}e_k$. The coefficients $a_{ik}$ define the structure matrix $A_\cE$ of $\cE$ relative to $B$ that codifies the dynamic structure of $\cE.$ 
These kind of algebras have numerous connections with other mathematical filed such as graph theory, group theory, Markov chains, dynamical systems, knot theory, 3-manifolds 
and the study of the Riemann-Zeta function (see \cite{T}).

In \cite{CSV,HA1,R1,R6,R4,R3} certain properties(such as nilpotency, derivations) of evolution algebras have been investigated.  In \cite{EL} nilpotency of evolution algebras  has been 
studied by means of graphs.  
Recently, many author have  studied evolution algebras, whose structural matrices are identified by the coefficients of
inheritance of  some quadratic stochastic operators  (see for instance \cite{LLR,R3,R4,R5,R7,R8} ). However, they have studied the properties of such algebra from algebraic point of view,
and has no implementation of dynamical behavior of such kind of operators even some dynamical properties have been carried out. Therefore, it is very natural to find some connections between 
the evolution algebra and the associated dynamical system ( see \cite{R7,R9}).  
In the present paper, we are gong to clarify this issue in the class of Volterra evolution algebras which we are going to be introduced. Namely, by looking some properties of Volterra 
evolution algebras,  we could predict dynamical behavior of  associated Volterra QSO. We notice that every Volterra QSO defines a genetic algebra (in sense of \cite{lu}) whose some properties have been investigated in \cite{G,GMPQ}. We
point out that our new algebra is not related to these types of algebras. 

The paper is organized as following. In section 2. we introduce Volterra evolution algebras, and prove that such kind of algebras are not nilpotent. We notice that 
in the literature mostly nilpotent evolution algebras have been investigated \cite{CSV,R1}. In section 3, we describe  isomorphism of some Volterra evolution algebras 
which will allow us to find a connect to the prediction of dynamical behavior of Volterra QSO. It is known that every Volterra QSO is associated to some weighted graph.  Furthermore, in section 4 we study 
the associated weighted graphs and Volterra evolution algebras, and prove that the graphs are isomorphic if and only if the corresponding Volterra algebras are isomorphic. 
Moreover, in the final section 5, we prove that derivation of Volterra evolution algebras whose graph is complete, is trivial. Moreover, the description of all derivations on 3-dimensional Volterra evolution algebras is provided.
       
\section{Voltera Quadratic Stochastic Operators}

In this section we recall a definition and some basic properties of Volterra quadratic stochastic operators.
 Let $$S^{m-1}=\{\textbf{x}=(x_1,x_2,...,x_m)\in \mathbb{R}^m \ : x_i\geq 0,\ \sum_{i=1}^{m}x_i=1\}$$ be the $m-1$ dimensional simplex.
A mapping $V:S^{m-1}\to S^{m-1}$ defined by
\begin{equation}\label{qso}
(Vx)_k := \sum_{i,j=1}^{m}p_{ij,k}x_ix_j, \ \ \  \ \ \ \ \ \ \forall k=\overline{1,m}
\end{equation}
is said to be a \textit{quadratic stochastic operator  (QSO)}  where $p_{ij,k}\geq0,$ $\sum_{k=1}^{m}p_{ij,k}$ and $p_{ij,k}=p_{ji,k}\ \forall i, j=\overline{1,m}.$

A QSO \eqref{qso} is called \textit{Volterra} if $p_{ij,k}=0 $ whenever $k\notin \{i,j\}$ for any $i,j\in\{1,m\}.$

We notice that a biological meaning of the Volterra condition is obvious, i.e. the offspring
repeats one of its parents genotype. 

It is known \cite{G} that any Volterra QSO can be
written in the following form 
\begin{equation}\label{volt}
(Vx)_k:=x_k\bigg(1+\sum_{i=1}^{m}a_{ki}x_i\bigg), \ \ k=1,\cdots,m,
\end{equation}
where $A_m=(a_{ki})_{k,i=1}^{m}$  is a skew-symmetric matrix with 
$|a_{ki}|\leq 1.$ 
Dynamics of such kind of operators was studied in \cite{G}. We refer to \cite{GMR} for general review about the theory of QSO. 

On the basis of numerical calculations Ulam conjectured \cite{ulma} that
the ergodic theorem holds for any QSO $V$ , that is, the
 $$\lim_{n\to \infty}\frac{1}{n}\sum_{k=0}^{n-1}V^kx$$  exists for any
$x \in S^{m-1}$. In 1977 Zakharevich \cite{Zak} proved that this conjecture is false, in general. He constructed an example of Volterra QSO which is not ergodic. 
Later on, in \cite{Nasr}, this result has been extended to general Volterra QSO in $S^2$ given by
\begin{equation}\label{2d}
V:\left\{
\begin{array}{l}
x_1'=x_1(1+ax_2-bx_3)\\
x_2'=x_2(1-ax_1+cx_3)\\
x_3'=x_3(1+bx_1-cx_2)
\end{array} \right.
\end{equation}
Namely, the following result has been established.

\begin{thm}\label{main} \cite{Nasr}
If the parameters $a, b, c$ for the Volterra quadratic stochastic operator \eqref{2d} have the
same sign and each of them non-zero, then the ergodic theorem will fail for this operator.
\end{thm}

Unfortunately, up to now, there is no general theorem to clarify the non-ergodicity of any Volterra QSO in higher dimensions.  Some particular cases  (in low dimensions) have been investigated in \cite{GGJ}. 
Therefore, it would be natural at least to predict
non-ergodic behavior of Volterra QSO in higher dimensions. 

\section{Volterra Evolution Algebra }

Let us first recall the definition of evolution algebra. 

\begin{defn} Let $\cE$ be a vector space over a field $\mathbb K$ with multiplication
 $\cdot$ and a basis $\{ e_1, e_2, . . ., e_n \}$ such that $$e_i\cdot e_j= 0, \ i\neq j,$$
$$e_i\cdot e_i=\sum_{k=1}^n a_{ik} e_k, \ i \geq 1,$$ then $E$ is called \textit{evolution algebra} and
basis  $\{e_1,  e_2, . . ., e_n \}$ is said to be \textit{natural basis}. \end{defn}

From the above definition it follows that evolution algebras are commutative (therefore, flexible).

The matrix $A=(a_{ij})_{i,j=1}^{n}$ is called \textit{matrix of the algebra}
$\cE$ in natural basis $\{e_1,\dots ,e_n\}.$  If $A=(a_{ij})_{i,j=1}^{n}$ is a skew symmetric matrix, then this kind of  evolution algebra is called 
\textit{Voltera evolution algebra}.  Let $\bf{\cE}$ be a vector space over a field $\mathbb K$.
In what follows, we always assume that $\mathbb K$ has characteristic zero.
The conical form of  the table of multiplication of  
 Volterra evolution algebra $\bf{\cE}$ w.r.t. {\it natural basis} $\{ e_1,  e_2, . . ., e_n \}$ is given by
\begin{eqnarray}
&& e_i\cdot  e_j= 0,\ i\neq j; \\ \label{zero}
&&e_i\cdot  e_i=-\sum_{k<i}^{i-1}a_{ki} e_k+\sum_{k>i}^{n}a_{ik} e_k. 
\end{eqnarray}
We note that if $i=1$ then the first part of \eqref{zero} is zero.

In what follows, by $A=(a_{ij})^n_{i,j=1}$ we mean the matrix of the structural constants
 of the finite-dimensional Volterra evolution algebra $\bf{\cE}$, which is skew symmetric.
Obviously, $rank A =\dim(\bf{\cE}\cdot\bf{\cE})$. Hence, for finite-dimensional evolution algebra the
rank of the matrix does not depend on choice of natural basis. An Volterra evolution algebra is non-degenerate if $e_ie_i\neq 0$ for any $i.$
In what follows,  we will consider non-trivial Voltera evolution algebra and for convenience, we write ${\bf u}{\bf v}$ instead ${\bf u}\cdot{\bf v}$
for any ${\bf u},{\bf v}\in\bf{\cE}$ and we shall write $\bf{\cE}^2$ instead $\bf{\cE}\cdot\bf{\cE}$.

A linear map $\psi: {\bf \cE}_1\to{\bf \cE}_2$ is called an {\it homomorphism} of  evolution algebras
if $\psi({\bf u}{\bf v})=\psi({\bf u})\psi({\bf v})$ for any ${\bf u},{\bf v}\in{\bf \cE}_1$. Moreover, if $\psi$ is bijective, then it is called an
{\it isomorphism}. In this case, the last relation is denoted by ${\bf \cE}_1\cong{\bf \cE}_2$.\\
\vspace{2mm}

\section{Nilpotency of Volterra Evolution Algebras}

In this section we are going to establish that any Volterra Evolution Algebra is not nilpotent. We notice that many existing results  are devoted to nilpotent  on evolution algebras (see for example, \cite{EL, HA1,R5}).

Given a non-associative algebra $\bf{\cE}$, we introduce the following sequences of subspaces:
\begin{align*}
\bf{\cE}^{<1>} & = \bf{\cE},&  \bf{\cE}^{<k+1>} & = \bf{\cE}^{<k>}\bf{\cE};\\[-4pt]
\bf{\cE}^1 & = \bf{\cE}, & \bf{\cE}^{k+1} &=\sum_{i=1}^{k}\bf{\cE}^i\bf{\cE}^{k+1-i}.
\end{align*}

\begin{defn} An algebra $\bf{\cE}$ is called
\begin{romanenumerate}
\item \emph{right nilpotent} if there exists $n\in \NN$ such that $\bf{\cE}^{<n>} = 0$, and the minimal 
such number is called the \emph{index of right nilpotency};
\item \emph{nilpotent} if there exists $n\in \NN$ such that $\bf{\cE}^n = 0$, and the minimal such number 
is called the \emph{index of nilpotency}.
\end{romanenumerate}
\end{defn}
\begin{rk}
A commutative algebra is right nilpotent if and only if it is nilpotent (see \cite[Chapter 4, Proposition 1]{ZSSS}). This applies, in particular, to evolution algebras.
\end{rk}

\begin{defn}
Let $\bf{\cE}$ be an algebra. 
Consider the chain of 
ideals $\ann^i(\bf{\cE})$, $i\geq 1$, where:
\begin{itemize}
\item
 $\ann^1(\bf{\cE})\bydef\ann(\bf{\cE})\bydef\{x\in\bf{\cE}: x\bf{\cE}=\bf{\cE} x=0\}$,
\item
$ \ann^i(\bf{\cE})$ is defined by $\ann^i(\bf{\cE}) /\ann^{i-1}(\bf{\cE})\bydef\ann(\bf{\cE} /\ann^{i-1}(\bf{\cE}))$.
\end{itemize}
The chain of ideals:
\[
0=\ann^0(\bf{\cE})\subseteq \ann^1(\bf{\cE})\subseteq \cdots\subseteq \ann^r(\bf{\cE})\subseteq \cdots
\] 
is called the  \emph{the upper annihilating series}.
\end{defn}

As for Lie algebras, a nonassociative algebra $\bf{\cE}$ is nilpotent if and only if its upper annihilating series reaches $\bf{\cE}$. 
That is, if there exists $r$ such that $\ann^r(\bf{\cE})=\bf{\cE}$.

In the following result shows that any Voltera evolution algebra is not nilpotent.   

\begin{thm}
Let $\bf{\cE}$ be a Volterra evolution algebra then $\bf{\cE}$ is not nilpotent.
\end{thm}

\begin{proof}
Let $B_{\bf{\cE}}=\{e_1,e_2,\ldots,e_n\}$ be a natural basis of $\bf{\cE}.$
 If $\bf{\cE}$ is non-degenerate then  $e_i^2\neq 0$  for any $1\leq i\leq n $ accordingly,  $ann(\bf{\cE})=0.$ We next claim  $ann^r (\bf{\cE})=0$ for any $r\in \NN.$  The proof is by induction on $r.$
For $r=2,$ using the fact  $e_i \in ann^2 (\bf{\cE})$ if and only if  $e_i^2\in ann(\bf{\cE}),$ but $e_i^2\neq 0$ for any $i,$ consequently  $e_i\notin ann^2 (\bf{\cE})$  for all $1\leq i\leq n $, so  $ann^2 (\bf{\cE})=0.$\\
For $r=k,$ we suppose our claim is true  that is  $ ann^k (\bf{\cE})=0,$ again by keeping in the mind the fact   $e_i\in ann^{(k+1)} (\bf{\cE})$ if and only if $e_i^2\in ann^k (\bf{\cE})$ since $\bf{\cE}$ is non-degenerate, i.e,  $e_i^2\neq 0$ for any $i.$ Hence, $e_i\notin ann^{(k+1)} (\bf{\cE}),$  for any $1\leq i\leq n,$ then we deduced that $ann^r (\bf{\cE})=0$  for any $r\in \NN.$ Therefore,  in the considered case  $\bf{\cE}$ is not nilpotent .

Now let us assume that  $\bf{\cE}$ is  degenerate,  then there is $1\leq i_0\leq n$ such that $e_{i_0}^2=0$ 
Hence,  $e_{i_0} \in ann(\bf{\cE}).$
Let us define the following set 
$$M^{(r)}:=\{i :\ \  e_i\in ann^r (\bf{\cE})\}$$

Since $\bf{\cE}$ is non-trivial then there is $k$ such that $e_k^2\neq 0$ so $e_k\notin ann(\bf{\cE})$  but $A$ is skew symmetric matrix, that is   $a_{i_{0}k}=-a_{ki_0 }=0$ 
for any $i_0\in M^{(1) }.$  We claim  $e_k \notin ann^r (\bf{\cE})$ for any $r\in \NN.$ Let us proof it by induction, for $r=2$ suppose by contrary $e_k\in ann^2 (\bf{\cE})$ then $e_k^2\in ann(\bf{\cE})$
This means that $e_k^2=\sum_{i\in M^{(1) }}\alpha_i e_i $ so 
$$\sum_{j\notin M^{(1) }}^{n}a_{kj}e_j=\sum_{i\in M^{(1) }}\alpha_i e_i $$
 We obtain $ a_{kj}=0$ for any $j\notin M^{(1)}$  keep in your mind $a_{ki_0 }=0$ for any $i_0 \in M^{(1)}$
Then we have $e_k^2=0$ which is contradiction to our assumption. Therefore, $e_k \notin ann^2 (\bf{\cE}).$
Suppose our claim is true for $r=k,$ i.e.  $e_k\notin ann^k (\bf{\cE}).$ 
Let $r=k+1$ and suppose by contrary $e_k\in ann^{(k+1)} (\bf{\cE})$ then we must have $ e_k^2\in ann^k (\bf{\cE})$
This implies that $e_k^2=\sum_{i\in M^{(k) }}\alpha_i e_i $ 
keep in your mind $a_{i_0 k}=-a_{ki_0 }=0 $ for any $i_0\in M^{(k)}.$  
Then 
$$\sum_{j\notin M^{(k) }}^{n}a_{kj}e_j=\sum_{i\in M^{(k) }}\alpha_i e_i $$
 Thus, $a_{kj}=0$ for any $j\notin M^{(k) }$ consequently, $e_k^2=0$ which is contradiction. Therefore, $e_k\notin ann^{(k+1)} (\bf{\cE}).$ Then there is no $r\in \NN$ such that $ann^r (\bf{\cE})=\bf{\cE}.$ Which means that  $\bf{\cE}$ is not nilpotent, this completes the proof.
\end{proof}

\section{Isomorphism of Some Volterra Evolution Algebras}

In this section we study isomorphisms of Volterra evolution algebras.

\begin{pro}\label{z1}
Let $\bf{\cE}_1,\bf{\cE}_2$ be two Volterra evolution algebras which are given by the following structural matrices
$$
\bf{A_\cE}_1=\left(
\begin{array}{ccc>{\columncolor{blue!20}}cccc}
0 & {a}_{12} & \ldots & 0 &\ldots & {a}_{1n}\\
-{a}_{12} & 0 & \ldots & 0 &\ldots & {a}_{2n}\\
\vdots & \vdots & \ddots & \vdots & \vdots & \vdots\\
\rowcolor{blue!20}
0 & 0 &\ldots &  0&\ldots &0\\
\vdots & \vdots & \vdots & \vdots & \vdots & \vdots \\
-{a}_{1n} & -a_{2n} & \ldots & 0 & \ldots & 0 
\end{array}\right)
$$ 
where $a_{ij}\neq 0$ for all $i\neq j, \ i,j\neq k $ 
$$
\bf{A_\cE}_2=\left(
\begin{array}{ccccc>{\columncolor{blue!20}}c}
0 & a'_{12} & \ldots & a'_{1,k+1} &\ldots & 0\\
-a'_{12} & 0 & \ldots & a'_{2,k+1} &\ldots & 0\\
\vdots & \vdots & \ddots & \vdots & \vdots & \vdots\\
-a'_{1,k+1} & -a'_{2,k+1} &\ldots &  0&\ldots &0\\
\vdots & \vdots & \vdots & \vdots & \vdots & \vdots \\
\rowcolor{blue!20}
0 & 0 & \ldots & 0 & \ldots & 0 
\end{array}\right)
$$ 
where $a'_{ij}\neq 0$ for all $1\leq i\neq j<n $. 
Then $\bf{\cE}_1 \cong \bf{\cE}_2$
\end{pro}
\begin{proof}
Let us defined $\psi : \bf{\cE}_2\to \bf{\cE}_1$ as follows 
$$
\psi(f_i)= \begin{dcases}
\frac{e_i}{a_{12}} & if \  i\neq m,\ 1\leq i<n; \\
\frac{e_n}{a_{12}} & if \  i=m; \\
 \frac{e_m}{a_{12}}& if \ i=n.
\end{dcases}
$$
Then it obviously $\psi$ is an isomorphism between $\bf{\cE}_1$ to $\bf{\cE}_2.$
\end{proof}
\begin{pro}\label{D}
Let $\bf{\cE}_1,\bf{\cE}_2$ be two Voltera evolution algebras which are given be the following structural matrices
\newcommand\hlight[1]{\tikz[overlay, remember picture,baseline=-\the\dimexpr\fontdimen22\textfont2\relax]\node[rectangle,fill=blue!50,rounded corners,fill opacity = 0.2,draw,thick,text opacity =1] {$#1$};}
$$
\bf{\cE}_1=\left(
\begin{array}{cccc}
0 & \hlight {0} & \ldots &  \hlight {a}\\
\hlight {0} & 0 & \ldots &   b\\
\vdots & \vdots & \ddots & \vdots\\
\hlight {-a} & -b & \ldots & 0  
\end{array}\right)
$$ 
where $ab\neq 0.$
$$
\bf{\cE}_2=\left(
\begin{array}{cccc}
0 & \hlight {1} & \ldots &   \hlight {0}\\
\hlight {-1} & 0 & \ldots &   b'\\
\vdots & \vdots & \ddots & \vdots\\
\hlight {0} & -b' & \ldots & 0  
\end{array}\right)
$$ 
where $b'\neq 0.$
Then $\bf{\cE}_1 \cong \bf{\cE}_2$ 
\end{pro}

 \begin{proof}
Let us defined $\psi : \bf{\cE}_2\to \bf{\cE}_1$ as follows 
$$
\psi(f_i)= \begin{dcases}
\frac{e_i}{a} & if \  i\neq 2,\ 1\leq i<n; \\
\frac{e_n}{a} & if \  i=2; \\
 \frac{e_2}{a}& if \ i=n.
\end{dcases}
$$
Then it obviously $\psi$ is an isomorphism between $\bf{\cE}_1$ to $\bf{\cE}_2.$
\end{proof}

Now, let us turn to the main aim of this paper which sheds some light into a relation between Volterra evolution algebra and dynamics of Volterra QSO.

Let $A$ be an $n\times n$ skew symmetric matrix. 
Let us define the following block skew symmetric matrices $m\times m$ of $A.$ Let $A_{m \times m}$ be  the leading principle matrix of order $m.$
By $A_{[j,j]}^{[i,i]}$ we denote $m\times m$ a skew symmetric matrix, which is obtained from $A_{m \times m}$ by deleting $i^{th}$ row and column, and adding $j^{th}$ row and  column from $A$ to the $m^{th}$ row and column, respectively, to 
the obtained one.  

Suppose that we have $n-$dimensional Volterra evolution  algebra $\cE$ with structural matrix $A_\cE$  with $rank(A)=m,$ and let $a_{ij}\neq 0$ for any $1\leq i\neq j \leq n.$  Fix $k$ such that $m<k \leq n.$
 Then 
 $$e_k^2=A_1^{(k)} e_1^2+A_2^{(k)} e_2^2+\ldots+A_m^{(k)} e_m^2,$$ where 
 $$A_i^{(k)}=\sqrt{\frac{det(A_{[k,k]}^{[i,i]})}{det(A_{m \times m})}}$$
 
 \begin{rk}
 Since $rank(A)=m$ then there is a least one skew symmetric matrix of order $m$ such that its determinate not zero, hence without loss of generality,  we may always assume $det(A_{m \times m}) )\neq 0.$ 
 \end{rk} 
 
 Let us consider some concrete example. 
 
 \begin{ex}
 Let $A_{6\times6}$ be a skew symmetric matrix with rank $4.$  
Then the leading principle skew symmetric matrix of order $4$ as follows 
 $$
A_{4\times 4}=\left(
\begin{array}{cccc}
0 & a_{12} & a_{13} &   a_{14}\\
-a_{12} & 0 & a_{23} &   a_{24}\\
-a_{13} & -a_{23} & 0 & a_{34}\\
-a_{14} & -a_{24} & -a_{34} & 0  
\end{array}\right)
$$ 
Now,  
$A_{[5,5]}^{[1,1]}$  is a $4\times 4$ skew symmetric matrix, which is obtained from $A_{4 \times 4}$ by deleting $1^{th}$ row and column and adding the rest entries from $5^{th}$ row and  column of $A_{6\times 6}$ as follows 
 $$
A_{[5,5]}^{[1,1]}=\left(
\begin{array}{cccc}
0 & a_{23} & a_{24} &   a_{25}\\
-a_{23} & 0 & a_{34} &   a_{35}\\
-a_{24} & -a_{34} & 0 & a_{45}\\
-a_{25} & -a_{35} & -a_{45} & 0  
\end{array}\right)
$$ 
Now, if we have $6-$ dimensional Voltera evolution algebra with rank of $A$ is $4$, then 
we can write $e_5^2=A_1^{(5)}  e_1^2+A_2^{(5)} e_2^2+A_3^{(5)} e_3^2+A_4^{(5)}  e_4^2$
where $$A_1^{(5)}=\sqrt{\frac{det(A_{[5,5]}^{[1,1]})}{det(A_{4\times 4})}}, \ \ A_2^{(5)}=\sqrt{\frac{det(A_{[5,5]}^{[2,2]})}{det(A_{4\times 4})}},$$
$$A_3^{(5)}=\sqrt{\frac{det(A_{[5,5]}^{[3,3]})}{det(A_{4\times 4})}}, \ \ A_4^{(5)}=\sqrt{\frac{det(A_{[5,5]}^{[4,4]})}{det(A_{4\times 4})}}.$$
 \end{ex}
 
 \begin{thm}\label{isom1}
 Let $\cE_1$ and $\cE_2$ be two Volterra evolution algebras with $a_{ij}b_{lm}\neq 0 $ for any $1\leq i\neq j \leq n,\ 1\leq l\neq m \leq n$. 
 Then $\cE_1\cong \cE_2 $ if and only if $\frac{a_{ij}}{b_{ij}} =\frac{a_{lm}}{b_{lm}}.$ 
 \end{thm}
 \begin{proof}
 $(\Rightarrow)$ This part of proof is divided into two cases depending on the rank of the structural matrix of $\cE_1$. Suppose by contrary, 
 there are $i_0, j_0, l_0, m_0$ such that $\frac{a_{i_0j_0}}{b_{i_0j_0}} \neq \frac{a_{l_0m_0}}{b_{l_0m_0}}.$
 
\textbf{Case} 1: Let $rank({\cE}_1)=n,$ as $\cE_1\cong \cE_2 $ then there is an isomorphism from $\cE_1$ onto $\cE_2$ defined as follows: $f_i=\sum_{j=1}^{n}\alpha_{ij}e_j.$ For any $i\neq j$ one has  $f_if_j=0,$ 
this implies that  $\sum_{k=1}^{n}\alpha_{ik}\alpha_{jk}e_k^2=0$. Due to $rank({\cE}_1)=n$ we obtain $\alpha_{ik}\alpha_{jk}=0$ for all $1\leq k\leq n$. Then without loss of generality, 
we may assume that $\alpha_{ii}\neq 0$ for any $1\leq i \leq n.$ So,  one gets $f_i=\alpha_{ii} e_i$, which yields $$f_i^2=\alpha_{ii}^2 \Bigg(\sum_{j<i}^{n}-a_{ji}e_j+\sum_{j>i}^{n}a_{ij}e_j \Bigg)$$
Then $\alpha_{ii}^2 a_{ij}=\alpha_{jj} b_{ij}.$  From $f_j^2=\alpha_{jj}^2 e_j^2$ we obtain $\alpha_{jj}^2 a_{ij}=\alpha_{ii} b_{ij}$,  this implies that $\alpha_{ii}=\alpha_{jj}$ for any $1\leq i \neq j \leq n.$
Since $\alpha_{ii}^2 a_{ij}=\alpha_{jj} b_{ij}$ then we have $\alpha_{ii}=\frac{b_{ij}}{a_{ij}}.$
 Now, let $i=i_0,\ j=j_0$ we have $\alpha_{i_0i_0}=\alpha_{j_0j_0}=\frac{b_{i_0j_0}}{a_{i_0j_0}}.$ Next  $i=l_0,\ j=m_0$ one finds $\alpha_{l_0l_0}=\alpha_{m_0m_0}=\frac{b_{l_0m_0}}{a_{l_0m_0}}$. Moreover, 
 $\alpha_{i_0i_0}=\alpha_{l_0l_0}$ since $\alpha_{ii}=\alpha_{jj}$ for any $1\leq i \neq j \leq n.$ This gives $\frac{b_{i_0j_0}}{a_{i_0j_0}}=\frac{b_{l_0m_0}}{a_{l_0m_0}}$ which contradicts to our assumption. 
 Therefore, in considered case  if $\cE_1\cong \cE_2 $ then $\frac{a_{ij}}{b_{ij}} =\frac{a_{lm}}{b_{lm}}.$
 
\textbf{Case} 2: Let $rank(\cE_1)=m<n,$ and  $x=\sum_{i=1}^{n}\alpha_ie_i$ be  a zero divisor in $\cE_1$ with $\prod_{i=1}^{n}\alpha_i\neq 0,$ i.e. there exists a non zero element $y=\sum_{i=1}^{n}\beta_ie_i$ in $\bf{\cE}_1$ such that $xy=0$. This together with the fact $e_k^2=\sum_{j=1}^{m}A_j^{(k)} e_j^2$, (for any  $m<k\leq n$) yields that 
\begin{equation}\label{3}
\alpha_i \beta_i+\sum_{l=m+1}^{n}\alpha_l \beta_l A_i^{(l)}=0  
\end{equation}
Since $y$ is non-zero, then we conclude that at least one of  $A_i^{(l)}$ is non zero. According to $\cE_1 \cong \cE_2 $ then there is an isomorphism  $\psi: \cE_1 \to \cE_2 $, hence we have  
$\psi(xy)=\psi(x)\psi(y)=\sum_{i=1}^{n}\alpha_i\beta_i(\psi(e_i))^2=0$, this gives the following system 
\begin{equation}\label{4}
 \alpha_i \beta_i+\sum_{l=m+1}^{n}\alpha_l \beta_l B_i^{(l)}=0
\end{equation}
From \eqref{3} and \eqref{4} one finds $\frac{A_i^{(p)}}{B_i^{(p)}}=\frac{A_i^{(q)}}{B_i^{(q)}}$ for any $m<p\neq q\leq n.$ This is possible only when $\frac{a_{ij}}{b_{ij}} =\frac{a_{lm}}{b_{lm}}$ 
for all $1\leq i\neq j \leq n,\ 1\leq l\neq m \leq n$ but this contradicts to our assumption.
Therefore, in considered case  if $\cE_1\cong \cE_2 $ then $\frac{a_{ij}}{b_{ij}} =\frac{a_{lm}}{b_{lm}}.$

$(\Leftarrow)$ By the following change of basis  
$f_i=\frac{a_{ij}}{b_{ij}}e_i,$
and tby means of the hypothesis we obtain $\cE_1 \cong \cE_2.$ This completes the proof. 
 \end{proof}
 
\begin{cor}\label{isoth}
Let $\cE_1$ and $\cE_2$ be two Volterra evolution algebras with the following structural matrices
$$
{\bf{A_\cE}}_1=\left(
\begin{array}{ccc}
0 & a_1 & -b_1\\
-a_1 & 0 & c_1 \\
b_1 & -c_1 & 0   
\end{array}\right)
,\ \ 
{\bf{A_\cE}}_2=\left(
\begin{array}{ccc}
0 & a_2 & -b_2\\
-a_2 & 0 & c_2 \\
b_2 & -c_2 & 0   
\end{array}\right) 
$$
respectively, where $a_1b_1c_1\neq 0, \ a_2b_2c_2\neq 0.$ Then $\cE_1 \cong \cE_2$ if and only if $\frac{a_2}{a_1}=\frac{b_2}{b_1}=\frac{c_2}{c_1}.$  
\end{cor}

The following theorem yields a relation between isomorphism of Volterra evolution algebra and the ergodicity of Voltera QSO in two dimensional setting.

\begin{thm}
Let $\cE_1$ and $\cE_2$ be 3-dimensional two isomorphic Volterra evolution algebras given as in corollary \ref{isoth}. Then the corresponding Volterra QSOs are  either ergodic or non ergodic. 
\end{thm}

\begin{proof}
Since $\cE_1 \cong \cE_2$ then by Theorem \ref{isoth}, we have $\frac{a_2}{a_1}=\frac{b_2}{b_1}=\frac{c_2}{c_1}.$   This implies that their structural matrices of either have the same signs or different signs. 
Hence, Theorem \ref{main} yields the required assertion. 
\end{proof}

From this theorem, we may formulate the following conjecture. 

\begin{con}
Let ${\cE}_1$ and ${\cE}_2$ be two $n-$dimensional isomorphic Voltera evolution algebras. Then the corresponding Volterra QSOs are either ergodic or non ergodic. 
\end{con}

In \cite{GGJ} it has been described dynamical behavior of all extremal Volterra QSO on low dimensional simplexes (up to dimension 5).   Now, the last conjecture allows to
predict dynamical behavior of other kinds of Volterra QSO which are isomorphic to extremal ones.  The next result supports this conjecture. 

\begin{ex}
Let us consider the following structural matrix of Volterra evolution algebra $\cE_1$:
$${\bf{A_\cE}}_1=\left(
\begin{array}{cccc}
0 & 1 & 1& -1\\
-1 & 0 & 1 & 1 \\
-1 & -1 & 0& 1\\
1 & -1 & -1& 0   
\end{array}\right) 
$$
in \cite{GGJ}, it has been showed that Volterra QSO associated with above given matrix is non-ergodic. Now let us consider another Volterra evolution algebra  $\cE_2$ with the following structure matrix   
$${\bf{A_\cE}}_2=\left(
\begin{array}{cccc}
0 & a & a& -a\\
-a & 0 & a & a \\
-a & -a & 0& a\\
a & -a & -a& 0   
\end{array}\right) 
$$
where $0\neq |a|<1$.  From Theorem \ref{isom1} we find that $\cE_1\cong \cE_2.$  Now, we are going to show that the corresponding Volterra QSO (to $\cE_2$) is not ergodic. Indeed, due to \eqref{volt} the Volterra
QSO has the following form:
 \begin{align}\label{ug}
V_{A_{\cE_2}} := \left\{ \begin{array}{cc} 
                 x_1'=x_1(1+ax_2+ax_3-ax_4) \\
                x_2'=x_2(1-ax_1+ax_3-ax_4) \\
                x_3'=x_3(1-ax_1-ax_2-ax_4) \\
                x_4'=x_4(1+ax_1-ax_2-ax_3) 
                \end{array} \right.
\end{align}
Introducing new variables $x=x_1,\ y=x_2+x_3,\ z=w$, one can see that the last QSO \eqref{ug} reduces to  
   \begin{align}\label{ug1}
\tilde V_{A_{\cE_2}} := \left\{ \begin{array}{cc} 
                 x'=x(1+ay-az) \\
                y'=y(1-ax+az) \\
                z'=z(1+ax-ay)                  
                \end{array} \right.
\end{align}
Then by Theorem \ref{main} we conclude that $\tilde V_{A_{\cE_2}}$ is not ergodic.  Hence, $V_{A_{\cE_2}}$ is not ergodic. 
\end{ex}    

Hence, the formulated conjecture allows us to predict dynamical behavior of Volterra QSO by examining algebraic structure of the corresponding Volterra evolution algebras. 

 \section{Graphs of Volterra Evolution Algebras}
 
 In this section, we are going to find a relation between the isomorphism of two Voltera evolution algebras and the isomorphism of the associated weighted graphs.  
 
 \begin{defn}
Let $\cE$ be an evolution algebra with a natural basis $B=\{e_1,\ldots,e_n\}$ and matrix of structural constants $A=\bigl(\alpha_{ij}\bigr)$.
\begin{itemize}
\item A graph $\Gamma(\cE,B)=(V,E)$, with $V=\{1,\ldots,n\}$ and $E=\{(i,j)\in V\times V: \alpha_{ij}\ne 0\}$, is called the \emph{graph attached to the evolution algebra $\cE$ relative to the natural basis $B$}.
\item The triple $\Gamma^w(\cE,B)=(V,E,\omega)$, with $\Gamma(\cE,B)=(V,E)$ and where $\omega$ is the map $E\rightarrow \FF$ given by $\omega\bigl((i,j)\bigr)=\alpha_{ij}$, is called the \emph{weighted graph attached to the Volterra 
evolution algebra $\cE$ relative to the natural basis $B$}.
\end{itemize}
 
\end{defn}
\begin{defn}
If every two vertices of a graph is connected by an edge, then such a graph is called \textit{complete}.  A complete graph with $n$ vertices is denoted by $K_n.$
\end{defn}

\begin{defn}
Let $\Gamma_1=(V_1,E_1,w_1)$ and $\Gamma_2=(V_2,E_2,w_2)$ be two weighted graphs. We call $\Gamma_1$ and $\Gamma_2$ are $w-$isomorphic $(\cong^{w})$ if the following conditions hold:
\begin{itemize}
\item[(i)] The corresponding unweighted graphs $(V_1,E_1)$ and $(V_2,E_2)$ are isomorphic, i.e. there is a bijection $\pi:V_1\to V_2$ with $w_1(i,j)\neq 0$, $w_2(\pi(i),\pi(j))\neq 0$ for all $i,j\in V_1$;
\item[(ii)] For non zero weights one has  $\frac{w_1(i,j)}{w_2(\pi(i),\pi(j))}=\frac{w_1(l,m)}{\w_2(\pi(l),\pi(m))}$, $i,j,l,m\in V_1$.
\end{itemize} 
\end{defn}

\begin{rk}
We stress that  the $w-$ isomorphism is an equivalence relation. Indeed, it is enough tp prove the transitive property i.e, if $\Gamma_1 \cong^{w} \Gamma_2  $ and $\Gamma_2 \cong^{w} \Gamma_3 $ then we have to show that $\Gamma_1 \cong^{w} \Gamma_3.$ Due to  $\Gamma_1 \cong^{w} \Gamma_2  $ then there is $\pi_1:V_1\to V_2$ such that $\frac{w_1(i,j)}{w_1(l,m)}=\frac{w_2(\pi_1(i),\pi_1(j))}{\w_2(\pi_1(l),\pi_1(m))} $, and  from 
$\Gamma_2 \cong^{w} \Gamma_3 $ there is $\pi_2:V_2\to V_3$ such that  $\frac{w_2(\pi_1(i),\pi_1(j))}{\w_2(\pi_1(l),\pi_1(m))}=\frac{w_3(\pi_3(i),\pi_3(j))}{\w_3(\pi_3(l),\pi_3(m))} $, where $\pi_3=\pi_2\circ \pi_1$. This implies 
$\frac{w_1(i,j)}{w_1(l,m)}= \frac{w_3(\pi_3(i),\pi_3(j))}{\w_3(\pi_3(l),\pi_3(m))} $ Hence, $\Gamma_1 \cong^{w} \Gamma_3.$ 
\end{rk}

\begin{thm}\label{nd}
Let $\cE_1$ and $\cE_2$ be two  $n-$ dimensional Volterra evolution algebras given as in theorem \ref{isom1} with associated weighted graphs  $\Gamma_1 , \Gamma_2$ respectively. Then  the following statements are equivalent:  
\begin{itemize}
\item[(i)] $\Gamma_1 \cong^w \Gamma_2$;
\item[(ii)] $\cE_1 \cong \cE_2$.
\end{itemize}
\end{thm}

\begin{proof}
(i)$\Rightarrow $ (ii)
Let $\pi:\Gamma_1\to \Gamma_2$ be an isomorphism.  Let us do the following change of basis:
 $f_i'=\frac{b_{\pi(i)}}{a_{\pi(i)}}e_{\pi(i)}.$ Then one can see the obtained  evolution algebra  $\cE'$ 
 we have
 $$ {f'_i}^2=\Bigg(\frac{b_{\pi(i)}}{a_{\pi(i)}}\Bigg)^2e^2_{\pi(i)}
$$
which is clearly Volterra evolution algebra.
Now, we are going to show that $\cE'\cong \cE_2,$ to end this job take the following change of basis $$f_i=f'_{\pi^{-1}(i)}$$ So by these change of basis we have $\cE'\cong \cE_2.$ Therefore,  $\cE_1\cong \cE_2 $\\

(i)$\Leftarrow $ (ii) Since $a_{ij}\neq 0$ and $b_{ij}\neq 0$ for any $1\leq i\neq j \leq n$ then the associated weighted graphs are completed graphs with the same number of vertices,  this means that the corresponding 
unweighed graphs are isomorphic. From 
$\frac{w_1(i,j)}{w_2(\pi(i),\pi(j))}=\frac{w_1(l,m)}{\w_2(\pi(l),\pi(m))}$, we obtain  $\frac{a_{ij}}{b_{ij}} =\frac{a_{lm}}{b_{lm}},$ for any $1\leq i\neq j \leq n,\ 1\leq l\neq m \leq n$, which, by Theorem \ref{isom1}, yields the assertion.  This completes the proof.
\end{proof}

\begin{cor}
Let $\cE_1$ and $\cE_2$ be  three dimensional Volterra evolution algebras, if their graphs are isomorphic as a graphs and one has  
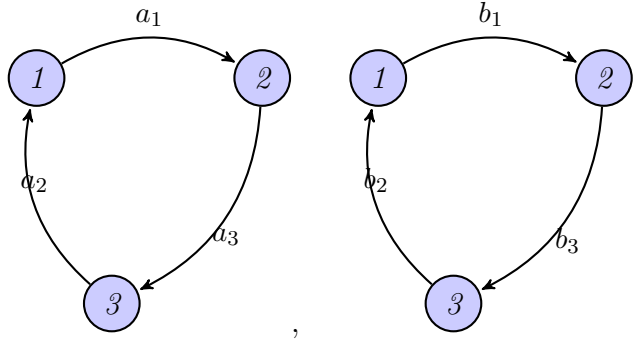
\[
\begin{tikzpicture}%
  %[->,auto=left,thick,every node/.style={circle,draw}]
[->,>=stealth',shorten >=1pt,auto,thick,every node/.style={circle,fill=blue!20,draw}]
  %\node (n6) at (1,10) {6};
  \node (n1) at (4,8)  {1};
  %\node (n5) at (8,9)  {5};
  \node (n2) at (7,8) {2};
  %\node (n2) at (9,6)  {2};
  \node (n3) at (5,5)  {3};

  %\foreach \from/\to in {n6/n4,n4/n5,n5/n1,n1/n2,n2/n5,n2/n3,n3/n4}
  %  \draw (\from) -- (\to);
   \path[every node/.style={font=\sffamily\small}]
%  (n1) edge node [above] {0.6} (n4)
     % (n2)  edge [bend right] node[left] {0.3} (n1)
           (n2)  edge [bend left] node[below] {$a_3$} (n3)
       %(n2) edge [loop right] node {$\delta$} (n2)
        (n1)  edge  [bend left] node[above] {$a_1$} (n2)
       %(n1) edge [loop left] node {$\alpha$} (n1)
       (n3)  edge  [bend left] node[above] {$a_2$} (n1); 
       
\end{tikzpicture}
,\ \ \ \  \begin{tikzpicture}%
  %[->,auto=left,thick,every node/.style={circle,draw}]
[->,>=stealth',shorten >=1pt,auto,thick,every node/.style={circle,fill=blue!20,draw}]
  %\node (n6) at (1,10) {6};
  \node (n1) at (4,8)  {1};
  %\node (n5) at (8,9)  {5};
  \node (n2) at (7,8) {2};
  %\node (n2) at (9,6)  {2};
  \node (n3) at (5,5)  {3};

  %\foreach \from/\to in {n6/n4,n4/n5,n5/n1,n1/n2,n2/n5,n2/n3,n3/n4}
  %  \draw (\from) -- (\to);
   \path[every node/.style={font=\sffamily\small}]
%  (n1) edge node [above] {0.6} (n4)
     % (n2)  edge [bend right] node[left] {0.3} (n1)
           (n2)  edge [bend left] node[below] {$b_3$} (n3)
       %(n2) edge [loop right] node {$\delta$} (n2)
        (n1)  edge  [bend left] node[above] {$b_1$} (n2)
       %(n1) edge [loop left] node {$\alpha$} (n1)
       (n3)  edge  [bend left] node[above] {$b_2$} (n1); 
       
\end{tikzpicture}
\]
 then the corresponding Volterra evolution algebras 
are isomorphic. Here $a_1,a_2,a_3$ have the same sign and $b_1,b_2,b_3$ have the same sign.
\end{cor}

\section{Derivation of Volterra evolution algebras and the associated graphs}

It is well known \cite{CGOT} that the derivation of any evolution algebra with non-singular matrix is zero.  As any skew symmetric matrix always has an even rank, therefore, if 
$n$ is odd then the maximal rank of our structural matrix could be  $n-1.$ In this section, we are going to describe derivation of $n$-dimensional Volterra evolution algebras whose associated graphs is complete. 

 In what follows, we need the following auxiliary fact.
 
\begin{lemma}\label{inv}
Let $D:\cE\to \cE$ be a derivation and suppose that the $rank(A_{\cE})=m$ where $m$ is even, then the following statements hold true
\begin{itemize}
\item[(i)] The subspace $\cE_1=\cE\ominus \cup_{j=m+1}^{n}\langle e_j \rangle$ is invariant under $D$ if $d_{ij}=0$ for all $1\leq i\leq m.$
\item[(ii)] The subspace $\cE_2=\cE\ominus \cup_{i=1}^{m}\langle e_i \rangle$ is invariant under $D$ if $d_{ji}=0$ for all $m+1\leq j\leq n.$
\end{itemize}
\end{lemma}
\begin{proof}
 Let $x\in \cE_1$ then $x=\sum_{i=1}^{m}\alpha_ie_i$. Then 
$$D(x)=\sum_{i=1}^{m}\alpha_i D(e_i)=\sum_{i=1}^{m}\alpha_i\sum_{k=1}^{n}d_{ik}e_k=\sum_{k=1}^{n}\left(\sum_{i=1}^{m}\alpha_i d_{ik}\right)e_k. $$ 
Due to $d_{ij}=0$ for all $1\leq i\leq m$, we have  
$$D(x)=\sum_{k=1}^{m}\left(\sum_{i=1}^{m}\alpha_i d_{ik}\right)e_k,$$which implies  $D(x)\in \cE_1.$ Hence $\cE_1$ is invariant under $D$.
 The proof of (ii) can be proceeded by the same argument as in (i). This completes the proof.
\end{proof}

\begin{pro}\label{odd}
 Let $\cE$ be a Volterra evolution algebra with  structural matrix $A_{\cE}$ such that $rank( A_{\cE})=n-1$. If the associated weighted graph $\Gamma$ is complete, then any derivation of $\cE$ is trivial. 
\end{pro}
\begin{proof}
First note that after suitable change of basis, we may assume that the first $n-1$ rows are linearly independent. So, $e_n^2=\sum_{i=1}^{n-1}\alpha_ie_i^2$  
 and due to the completeness of $\Gamma$, one has   $\prod_{i=1}^{n-1}\alpha_i\neq 0.$ Assume that $D$ is a derivation of $\cE$. By apply the derivation rule to $e_ie_n$, where $i\neq n,$  we have
 $$D(e_ie_n)=d_{in}e_n^2+d_{ni}e_i^2=0.$$
 Now plugging the value of $e_n^2$ into the last expression, one finds $$d_{in}\left( \sum_{i=1}^{n-1}\alpha_ie_i^2\right) +d_{ni}e_i^2=0$$
which gives the following system 
 $$\alpha_id_{in}+d_{ni}=0, \ \alpha_ld_{in}=0, \ l\neq i $$
According to $\prod_{i=1}^{n-1}\alpha_i\neq 0$,  we obtain $d_{in}=d_{ni}=0.$ Now, from (i) of Lemma \ref{inv}, it follows that $\cE_1=\cE\ominus <e_n>$ is invariant under $D$. Since 
$rank(A_{\cE_1})=n-1$  then $det(A_{\cE_1})\neq 0$ this implies that  $D_{\cE_1}$ is trivial. So, $spec(d)=\{d_{nn}\}.$ Due to  $e_n^2\neq 0$, we conclude that  $2d_{nn}$ is an eigenvalue of
$D$, hence, $spac(d)=\{d_{nn}\}=\{2d_{nn}\},$ which is possible if $D = 0.$ This proves the proposition.
\end{proof}

Now we are ready to prove a main result of this section. 

\begin{thm}\label{comgraph}
 Let $\cE$ be a Volterra evolution algebra with  structural matrix $A_{\cE}$. If the associated weighted graph $\Gamma$ is complete, then any derivation of $\cE$ is trivial. 
 \end{thm}
\begin{proof}
 Since $\Gamma$ is completed graph then $deg(v_i)=n-1.$ Hence the entries of structural matrix of constant $A_{\cE}$ are non-zero for any $  i\neq j$.
Suppose that  $rank(A_{\cE})=n,$ it follows that $det(A_{\cE})\neq 0$ consequently,  the derivation of $\cE$ is trivial \cite{CGOT}.

 If $rank(A_{\cE})=m<n,$ as  $A_{\cE}$ is a skew symmetric matrix then $m$ must be even. Therefore, after suitable permutation of the basis, we may assume that the first $m$ rows are linearly independent. So, if $m<j\leq n$ one can write 
 \begin{eqnarray}\label{li}
e_j^2=\alpha_1e_1^2+\alpha_2e_2^2+\ldots+\alpha_me_m^2/
\end{eqnarray} 
Take into account Theorem \ref{nd} one gets $\prod_{i=1}^{m}\alpha_i\neq 0.$ 
 If $D$ is a derivation of $\cE$, then we have 
$$D(e_ie_j)=d_{ij}e_j^2+d_{ji}e_i^2=0, \ \ 1\leq i\leq m<j\leq n.$$ 
Putting \eqref{li} into the last expression, one has 
 $$d_{ij}(\alpha_1e_1^2+\alpha_2e_2^2+\ldots+\alpha_me_m^2)+d_{ji}e_i^2=0$$    
 Therefore, we get the following system  
$$\alpha_i d_{ij}+d_{ji}=0,\ \ 
\alpha_l d_{ji}=0,\ 1\leq l\leq m ,\ l\neq i.$$
This implies that $d_{ij}=d_{ji}=0.$
Let us define $\cE_1=\cE\ominus \bigcup_{j=m+1}^{n}\langle e_j\rangle$. It is clear that $\cE_1$ has rank $m$ hence $det(A_{\cE_1})\neq 0$ thus the derivation of $\cE_1$ is trivial.  Using (ii) of Lemma \ref{inv} we have  $D_{\cE_1}=0.$ 
So,  $spec(d)=\{d_{jk}: \ m+1\leq i,k \leq n\}.$
Next, we define  $\cE_2=\cE\ominus \bigcup_{i=1}^{m}\langle e_i\rangle$ then the structural matrix $A_{\cE_2}$ is  an $n-m\times n-m$  skew symmetric matrix and by (ii) of Lemma \ref{inv}  if $n-m=1$ then the $rank(A_{\cE})=n-1$ so, by 
Proposition \ref{odd} we conclude that $d(\cE)$ is trivial. Now, consider $n-m>1$ then repeating the same process for $\cE_1$ many times until we reach to $\cE_k=\cE_2\ominus \cup_{p=1}^{n-2}\left\langle e_p \right\rangle 
$ whose the structural matrix $A_{\cE_k}$ is a $2\times 2$  skew symmetric matrix, and $\det(A_{\cE_k})\neq 0$. Hence $D_{\cE_k}$ is trivial. Consequently, Compiling all information together, we infer that $D$ is trivial, which proves the theorem.   
\end{proof}

We notice that in \cite{EL19} other classes of evolution algebras are discussed whose derivations are trivial.

%%%%%%%%%%%%%%%%%%%%%%%%%%%%%%%%%%%%%%%%%%%%%%%%%%%%%%%%%%%%%%%%%
 \begin{rk} From the last result we conclude that if the graph is completed then the derivation of the corresponding algebra is trivial. It would be interesting to know: if the graph is not completed, does there exist a non-trivial derivation on this algebra. 
 \end{rk}
 
 Now, let us turn to a particular case in order to describe full derivation in terms of graphs. 
 
 The next result describes derivation of three dimensional Volterra algebra in terms of graphs.
 
 \begin {thm}\label{tab}
 Let $\cE$ be a $3$-dimensional Volterra evolution algebra then its derivation aa  in the following table:
\begin{center}
%\captionof{table}{Describe the Derivation of Three Dimensional voltera Evolution Algebra }
\begin{tabular}{|lX|lX|c|}

\hline    Graph & Derivation  & $\dim$(Annihilator)   \\
\hline $\begin{tikzpicture}%
  %[->,auto=left,thick,every node/.style={circle,draw}]
[->,>=stealth',shorten >=-5pt,auto,thick,every node/.style={circle,fill=blue!20,draw}]
  %\node (n6) at (1,10) {6};
  \node (n1) at (6,8)  {1};
  %\node (n5) at (8.5,8.5)  {5};
  \node (n2) at (8,8) {2};
  %\node (n2) at (9,6)  {2};
  \node (n3) at (7,7)  {3};

  %\foreach \from/\to in {n6/n4,n4/n5,n5/n1,n1/n2,n2/n5,n2/n3,n3/n4}
  %  \draw (\from) -- (\to);
   \path[every node/.style={font=\sffamily\small}]
%  (n1) edge node [above] {0.6} (n4)
     % (n2)  edge [bend right] node[left] {0.3} (n1)
           (n2)  edge [bend left] node[below] {$a_3$} (n3)
       %(n2) edge [loop right] node {$\delta$} (n2)
        (n1)  edge  [bend left] node[above] {$a_1$} (n2)
       %(n1) edge [loop left] node {$\alpha$} (n1)
       (n3)  edge  [bend left] node[above] {$a_2$} (n1); 
       
\end{tikzpicture}$ &Trivial& 0   \\
 \hline $\begin{tikzpicture}%
  %[->,auto=left,thick,every node/.style={circle,draw}]
[->,>=stealth',shorten >=-5pt,auto,thick,every node/.style={circle,fill=blue!20,draw}]
  %\node (n6) at (1,10) {6};
  \node (n1) at (6,8)  {1};
  %\node (n5) at (8.5,8.5)  {5};
  \node (n2) at (8,8) {2};
  %\node (n2) at (9,6)  {2};
  \node (n3) at (7,7)  {3};

  %\foreach \from/\to in {n6/n4,n4/n5,n5/n1,n1/n2,n2/n5,n2/n3,n3/n4}
  %  \draw (\from) -- (\to);
   \path[every node/.style={font=\sffamily\small}]
%  (n1) edge node [above] {0.6} (n4)
     % (n2)  edge [bend right] node[left] {0.3} (n1)
           %(n2)  edge [bend left] node[below] {$a_3$} (n3)
       %(n2) edge [loop right] node {$\delta$} (n2)
        (n1)  edge  [bend left] node[above] {$a_1$} (n2)
       %(n1) edge [loop left] node {$\alpha$} (n1)
       (n3)  edge  [bend left] node[above] {$a_2$} (n1); 
       
\end{tikzpicture}$&$\left(
\begin{array}{ccc}
2\alpha & 0 & 0\\
0 & \alpha & 3a\alpha \\
0 & -3a^2\alpha & \alpha   
\end{array}\right)$, or  Trivial  
  & 0 \\
\hline
 $
\begin{tikzpicture}%
  %[->,auto=left,thick,every node/.style={circle,draw}]
[->,>=stealth',shorten >=1pt,auto,thick,every node/.style={circle,fill=blue!20,draw}]

  \node (n1) at (4,4)  {1};
  \node (n2) at (6,4)  {2};
  \node (n3) at (8,4)  {3};

   \path[every node/.style={font=\sffamily\small}]

           (n1) edge node [above] {$1$} (n2) ;

\end{tikzpicture}$& $\left(
\begin{array}{ccc}
0 & 0 & 0\\
0 & 0 & 0 \\
0 & 0 & \alpha   
\end{array}\right)$& 1\\
\hline

\end{tabular}
\end{center}
\end{thm}
\begin{proof}
Let $\cE$ be a three dimension Volterra evolution algebra with structural matrix of constant given by
 $$
\textbf{A}_\cE=\left(
\begin{array}{ccc}
0 & a_1 & a_2    \\
-a_1 & 0 &a_3\\
-a_2 & -a_3& 0 
\end{array}\right)
$$ 
as $\textbf{A}_{\cE}$ is a skew symmetric matrix, then its rank is $2.$ 
Then by proposition \ref{D} this algebra is isomorphic to 
 $$
\textbf{A}_\cE'=\left(
\begin{array}{ccc}
0 & 1 & a    \\
-1 & 0 &b\\
-a & -b& 0 
\end{array}\right)
$$ 
From this algebra one can easily find that $e_3^2=be_1^2+ae_2^2.$ 
Consider  $D(e_1e_2)=d_{12}e_2^2+d_{21}e_1^2=0$ due to the linearity independence of $e_1^2, e_2^2$ we have $d_{12}=d_{21}=0.$
Now, compute $D(e_1e_3)=(-bd_{13}+d_{31})e_1^2+ad_{13}=0$, hence we have $-bd_{13}+d_{31}=0,\ ad_{13}=0.$ By the same argument,  one can find from $D(e_2e_3)$ that $ad_{23}+d_{32}=0$, $-bd_{23}=0.$ 
So, we have the following system 
\begin{eqnarray}\label{case2}
\begin{dcases}
bd_{13}+d_{31}=0; \\
ad_{23}+d_{32}=0; \\
ad_{13}=0,\ \-bd_{23}=0.
\end{dcases}
\end{eqnarray}
Let us consider some cases

\textbf{Case} 1: If $ab\neq0,$ the from \eqref{case2} one gets $d_{23}=d_{32}=d_{13}=d_{31}=0.$
One can compute $D(e_1^2)=2d_{11}(e_2+ae_3)$. On the other hand $D(e_1^2)=d(e_2)+ad(e_3)=d_{22}e_2+ad_{33}e_3$, hence one finds $d_{22}=d_{33}=2d_{11}.$
Furthermore, evaluating $D(e_2^2)$ we have $d_{11}=d_{33}=2d_{22}$, and compiling all these information, we obtain $d_{11}=d_{33}=d_{22}=0.$ Thus, in this case the derivation is trivial.

\textbf{Case} 2: If $ab=0,$  and $a=b=0$. Then from \eqref{case2} we have $d_{31}=d_{32}=0.$ 
Considering $D(e_1^2)=d_{11}e_2,$  and from other side $D(e_1^2)=d(e_2)=d_{22}e_2+d_{23}e_3$, we get  $d_{23}=0,\ d_{22}=2d_{11}.$ Similarly evaluating $D(e_2^2)$  one gets
$d(e_2^2)=-d_{22}e_1$, and moreover, one has $D(e_2^2)=-d(e_1)=-d_{11}e_1+d_{13}e_3$, which implies $d_{13}=0,\ d_{11}=2d_{22}.$ So, finally we have $d_{11}=d_{22}=0.$ 
Hence the derivation in this case has the form as in the last row of the above table.

\textbf{Case} 3: If $ab=0,$ and one of  $a$ and $b$  is no zero. Then due to Proposition \ref{D} we may assume that $a\neq 0.$ Then from system \eqref{case2} one gets $d_{13}=d_{31}=0$, $ad_{23}+d_{32}=0$. By 
evaluating $D(e_i^2)$ for all $1\leq i\leq 3$,  we obtain 
\begin{eqnarray}\label{case3}
\begin{dcases}
d_{22}+ad_{32}=2d_{11};\\
d_{23}+ad_{33}=2ad_{11}; \\
ad_{23}+d_{32}=0; \\
2d_{11}=d_{22}=d_{33}.
\end{dcases}
\end{eqnarray}
Now the system $\eqref{case3}$ has non-trivial solution if and only if $a^3+1=0,$ So, if $a^3+1\neq 0$ then the derivation is zero . If $a^3+1=0$ then by solving \eqref{case3} we get the derivation as in row two of the above table.  This completes the proof
\end{proof}

\begin{rk}
The converse of Theorem \ref{comgraph} is not true, for example the table in Theorem \ref{tab} shows the existence trivial derivation, but the graph is not complete.
\end{rk}

\section*{acknowledgement}
The second named author (F.M.) thanks the UAEU grant Start-Up 2016 No. 31S259 for support.


\begin{thebibliography}{99}
 
 \bibitem{B0} S.N. Bernstein, Principe de stationarit\'{e} et g\'{e}n\'{e}ralisation de la loi de
Mendel, \textit{Comptes Rendus Acad. Sci. Paris}, {\bf 177}(1923),
581-584.

 \bibitem{CGOT} L.M. Camacho, J.R. Gomes, B.A. Omirov, R. M. Turdibaev, The derivations of some evolution algebras, \textit{Linear and Multilinear Algebra} {\bf 61}(2013), 309-322.


 \bibitem{R1} J.M. Casas,  M. Ladra,  U. A. Rozikov, A chain of evolution algebras. \textit{Linear Algebra and its Applications} {\bf 435}(2011): 852-870.

\bibitem{R5} J. M. Casas, M. Ladra, M., B.A. Omirov, U.A. Rozikov,   On evolution algebras. \textit{Algebra Colloquium} {\bf 21}(2014), 331-342. 

\bibitem{CM} T.~Cort\'es, F.~Montaner, On the structure of Bernstein algebras, \textit{J.~London Math.~Soc.}  \textbf{51} (1995), 41--52.

\bibitem{CSV} Y.~Cabrera Casado, M.~Siles Molina, M.V.~Velasco; 
Evolution algebras of arbitrary dimension and their decompositions, 
\textit{Linear Algebra Appl.} \textbf{495} (2016), 122--162.

%\bibitem{R7} A. Dzhumadil'daev, B.A. Omirov, U.A.  Rozikov,  (2014). On a class of evolution algebras of" chicken" population. International Journal of Mathematics, 25(08), 1450073.

\bibitem{R8}  A. Dzhumadil'daev, B. A. Omirov,  U. A. Rozikov, Constrained evolution algebras and dynamical systems of a bisexual population. \textit{Linear Algebra and its Applications} {\bf 496} (2016), 351-380. 

\bibitem{EL} A.~Elduque, A.~Labra, Evolution algebras and graphs, \textit{J.~Algebra Appl.} \textbf{14} 
(2015), no. 7, 1550103, 10 pp. 

\bibitem{EL19} A.~Elduque, A.~Labra, Evolution algebras, automorphisms and graphs, \textit{Linear and Multilinear Algebra} Doi: 10.1080/03081087.2019.1598931


\bibitem{E} I.M.H. Etherington, Genetic algebras. \textit{Proc. Roy. Soc. Edinburgh} {\bf 59} (1939), 242-258.

\bibitem{Nasr} N. N. Ganikhodzhaev, D.V. Zanin, On a necessary condition for the ergodicity of quadratic operators defined on the two-dimensional simplex, \textit{Russian Math. Surv.} {\bf 59} (2004), 571-572.

\bibitem{GGJ} N.  Ganikhodzhaev, R. Ganikhodzhaev, U. Jamilov,  Quadratic stochastic operators and zero-sum game dynamics, \textit{Ergodic Theor. Dyn. Syst.} {\bf 35} (2015),  1443-1473.

\bibitem{G} R.N. Ganikhodzhaev,  Quadratic stochastic operators, Lyapunov functions and tournaments. \textit{Russian Acad.
Sci. Sb. Math.} {\bf 76}(1993), 489-506.

\bibitem{GMPQ} R. Ganikhodzhaev, F. Mukhamedov, A. Pirnapasov, I. Qaralleh, On genetic Volterra algebras and their
derivations, \textit{Commun. Algebra} {\bf 46} (2018), 1353-1366.

\bibitem{GMR} R. Ganikhodzhaev, F. Mukhamedov, U. Rozikov,
Quadratic stochastic operators and processes: results and open
problems, \textit{Infin. Dimens. Anal. Quantum Probab. Relat. Top.}
{\bf 14}(2011), 270--335.


\bibitem{HA1} A.S.~Hegazi, H.~Abdelwahab, Nilpotent evolution algebras over arbitrary fields, 
\textit{Linear Algebra Appl.} \textbf{486} (2015), 345--360.

%\bibitem{HA2} A.S.~Hegazi, H.~Abdelwahab; \emph{Five-dimensional nilpotent evolution algebras},
%arXiv:1508.07442v1.

\bibitem{J}  N.~Jacobson, \textit{Lectures in Abstract Algebra}, vol.~I, Basic Concepts. D. Van Nostrand Company, Inc. 
1951.

\bibitem{LLR} A. Labra, M. Ladra,  U. A. Rozikov, An evolution algebra in population genetics, \textit{Linear Alg. Appl.} {\bf 457} (2014), 348-362.€

\bibitem{R6}  M. Ladra,  U. A. Rozikov, Evolution algebra of a bisexual population, \textit{J. of Algebra} {\bf 378} (2013), 153-172.‏


\bibitem{lu}  Yu. I. Lyubich, \textit{Mathematical structures in population genetics}, Biomathematics(Berlin) (1992).

\bibitem{NSE} S.G. Narendra, C.M. Samaresh, W.M.  Elliott, On the Volterra and other nonlinear moldes of interacting
populations, \textit{Rev. Mod. Phys.} {\bf  43}(1971), 231--276.

\bibitem{R3}  B.A. Omirov,  U. A. Rozikov,  K. M. Tulenbayev, On real chains of evolution algebras, \textit{Linear and Multilinear Algebra} {\bf 63}(2015): 586-600. €

\bibitem{Reed} M.L. Reed, Algebraic structure of genetic inheritance, \textit{Bull. Amer. Math. Soc.(N.S.)} {\bf 34}(1997), 107-130.

%\bibitem{R2} U.A. Rozikov,  Sh N. Murodov. "Chain of evolution algebras of "€œchicken€ population",  Linear Algebra and its Applications 450 (2014): 186-201.

\bibitem{R4} U.A. Rozikov,  Sh N. Murodov, Dynamics of two-dimensional evolution algebras, \textit{ Lobachevskii Jour. Math.} {\bf 34}(2013), 344-358. 

\bibitem{R9} U.A. Rozikov,  R. Varro, Dynamical systems generated by a gonosomal evolution operator, \textit{Discontinuity, Nonlinearity, and Complexity} {\bf 5} (2016): 173-185.

\bibitem{R7} U.A. Rozikov, M. V. Velasco, A discrete-time dynamical system and an evolution algebra of mosquito population, \textit{J. Math. Biol.} {\bf 78} (2019), 1225-1244.

\bibitem{TV} J.P.~Tian and P.~Vojtechovsky, Mathematical concepts of evolution algebras 
in non-Mendelian genetics,
\textit{Quasigroups Related Systems} \textbf{14} (1), (2006), 111--122.

\bibitem{T} J.P.~Tian; Evolution algebras and their applications.
Lecture Notes in Mathematics \textbf{1921}, Springer-Verlag, Berlin, 2008.


\bibitem {ulma} ~S.M. Ulam,  \textit{Collection of mathematical problems}. Vol. \textbf{8}. Interscience Publishers, 1960.


\bibitem{WB} A. Worz-Busekros,  \textit{Algebras in Genetics},
Lect. Notes in Biomathematics, Vol. 36, Springer-Verlag, Berlin,
1980.

\bibitem{Zak} M.I. Zakharevich, On the behaviour of trajectories and the ergodic hypothesis for quadratic mappings of a simplex, \textit{Russian Math. Surv.} {\bf 33}(1978), 265-266.

\bibitem{ZSSS} K.A.~Zhevlakov, A.M.~Slin'ko, I.P.~Shestakov, A.I.~Shirshov; Rings that are nearly associative. 
Pure and Applied Mathematics \textbf{104}. Academic Press 1982.

\end{thebibliography}
\end{document}